\documentclass[11pt,letterpaper,reqno]{amsart}
\usepackage[english]{babel}
\usepackage[applemac]{inputenc}
\usepackage[T1]{fontenc}

\usepackage{bbm}
\usepackage{float, verbatim}
\usepackage{amsmath,amssymb,amsthm,amsfonts,mathtools}
\usepackage{graphicx}
\usepackage
{hyperref}
\usepackage{color}
\usepackage{tikz}
\usepackage{xcolor}
\usepackage{caption}
\usepackage{subcaption}
\usepackage{enumerate}
\usepackage{url}

\hypersetup{
	colorlinks=true,
	linkcolor=blue,
	filecolor=magenta,      
	urlcolor=cyan,
	citecolor= red,
}

\newcommand{\Haus}{\dim_{\mathrm{H}}}

\numberwithin{equation}{section}
\numberwithin{figure}{section}

\newtheorem*{thm*}{Theorem}
\newtheorem*{conj*}{Conjecture}
\newtheorem*{ques*}{Question}
\newtheorem*{rem*}{Remark}
\newtheorem*{defn*}{Definition}
\newtheorem*{mainques*}{Main questions}

\newtheorem{thm}{Theorem}[section]

\newtheorem{lem}[thm]{Lemma}
\newtheorem{cor}[thm]{Corollary}

\newtheorem{conj}[thm]{Conjecture}

\theoremstyle{definition}
\newtheorem{defn}[thm]{Definition}
\newtheorem{rem}[thm]{Remark}

\newtheorem{ex}[thm]{Example}

\def\supp{\mathrm{supp}}

\def \dist {{\mathrm{dist}}}

\def \bxi {{\boldsymbol{\xi}}}

\def \bN {{\mathbb N}}

\def \bR {{\mathbb R}}
\def \bZ {{\mathbb Z}}

\def \cM {{\mathcal{M}}}
\def \cR {{\mathcal{R}}}
\def \cS {{\mathcal{S}}}

\def \del {{\delta}}
\def \eps {{\varepsilon}}
\def \epsilon {{\varepsilon}}
\def \lam {{\lambda}}

\def \bk {{\mathbf{k}}}
\def \bt {{\mathbf{t}}}

\def \bzero {{\boldsymbol{0}}}

\def \d {{\mathrm{d}}}

\def \bN {\mathbb N}
\def \bP {\mathbb P}

\def \bR {\mathbb R}

\def \bZ {\mathbb Z}

\def \ba {\mathbf a}

\def \bd {\mathbf d}

\def \bi {\mathbf i}

\def \bk {\mathbf k}

\def \bn {\mathbf n}

\def \bt {\mathbf t}

\def \bx {\mathbf x}

\def \by {\mathbf y}
\def \bz {\mathbf z}

\def \bzero {\mathbf 0}

\def \bgam {{\boldsymbol{\gam}}}

\def \bxi {{\boldsymbol{\xi}}}

\def \bgam {\boldsymbol{\gamma}}

\def \fd {\mathfrak d}

\def \fB {\mathfrak B}
\def \fC {\mathfrak C}

\def \cB {\mathcal B}

\def \cD {\mathcal D}
\def \cE {\mathcal E}
\def \cF {\mathcal F}

\def \cK {\mathcal K}

\def \cM {\mathcal M}
\def \cN {\mathcal N}

\def \cR {\mathcal R}
\def \cS {\mathcal S}

\def \cU {\mathcal U}

\def \le {\leqslant}
\def \leq {\leqslant}
\def \ge {\geqslant}
\def \geq {\geqslant}

\def \dim {\mathrm{dim}}
\def \dimh {{\mathrm{\dim_H}}}

\def \supp {{\mathrm{supp}}}

\def \d {{\mathrm{d}}}

\def \ds1 {\mathds{1}}

\def \del {{\delta}}
\def \eps {{\varepsilon}}
\def \kap {{\kappa}}
\def \lam {{\lambda}}

\usepackage[colorinlistoftodos]{todonotes}

\title[Diophantine approximation on fractals]{Simultaneous and multiplicative Diophantine approximation on missing-digit fractals}

\author{Sam Chow}
\address{Sam Chow, Mathematics Institute, Zeeman Building, University of Warwick, Coventry CV4 7AL, UK}
\curraddr{}
\email{Sam.Chow@warwick.ac.uk}

\author{Han Yu}
\address{Han Yu, Mathematics Institute, Zeeman Building, University of Warwick, Coventry CV4 7AL, UK}
\curraddr{}
\email{han.yu.2@warwick.ac.uk}

\thanks{}



\begin{document}

\makeatletter
\providecommand\@dotsep{5}
\makeatother

\begin{abstract}
We investigate the metric theory of Diophantine approximation on missing-digit fractals. In particular, we establish analogues of Khinchin's theorem and Gallagher's theorem, as well as inhomogeneous generalisations.
\end{abstract}

\maketitle

\section{Introduction}

\subsection{Two foundational results in metric Diophantine approximation}

We begin by recalling the foundational results of Khinchin and Gallagher. We write $\lam_k$ for $k$-dimensional Lebesgue measure, and write $\| x \|$ for the distance between a real number $x$ and the set of integers. For $\psi: \bN \to [0,1)$, we denote by $W_k^\times(\psi)$ the set of $\bx = (x_1, \ldots, x_k) \in [0,1)^k$ such that
\[
\| n x_1 \| \cdots \| n x_k \| < \psi(n)
\]
has infinitely many solutions $n \in \bN$. Similarly, we denote by $W_k(\psi)$ the set of $\bx \in [0,1)^k$ such that
\[
\max\{ \| n x_1 \|, \dots ,\| n x_k \|\} < \psi(n)
\] 
holds for infinitely many $n \in \bN$. 

\begin{thm}[Khinchin's theorem \cite{Khi1924, Khi1926}]
Let $\psi: \bN \to [0,1)$ be non-increasing, and let $k \in \bN$. Then
\[
\lam_k(W_k(\psi))
= \begin{cases}
0, &\text{if } \displaystyle
\sum_{n=1}^\infty 
\psi(n)^k
< \infty \\ \\
1, &\text{if } \displaystyle
\sum_{n=1}^\infty 
\psi(n)^k
= \infty.
\end{cases}
\]
\end{thm}

\begin{thm}[Gallagher's theorem \cite{Gal1962}]
Let $\psi: \bN \to [0,1)$ be non-increasing, and let $k \in \bN$. Then
\[
\lam_k(W_k^\times(\psi))
= \begin{cases}
0, &\text{if } \displaystyle
\sum_{n=1}^\infty 
\psi(n) (\log n)^{k-1} 
< \infty \\ \\
1, &\text{if } \displaystyle
\sum_{n=1}^\infty 
\psi(n) (\log n)^{k-1} 
= \infty.
\end{cases}
\]
\end{thm}

Khinchin's theorem gives the simultaneous approximation rate of a generic vector, whilst Gallagher's theorem gives the generic multiplicative approximation rate. The latter is a strong form of a famous conjecture of Littlewood, except that the approximation rate is only valid for almost all vectors.

\begin{conj}
[Littlewood's conjecture in $k$ dimensions]
Let \mbox{$\bx \in \bR^k$,} where $k \ge 2$. Then
\[
\liminf_{n \to \infty}
n \| n x_1 \| \cdots \| n x_k \| = 0.
\]
\end{conj}

\subsection{Diophantine approximation on fractals}
\label{sec: DAfrac}

Now that we have seen the classical results of Khinchin and Gallagher, it is natural to ask what happens if one replaces Lebesgue measure with some other measure. A probability measure $\mu$ is \emph{Khinchin} if Khinchin's theorem holds with $\mu$ in the place of $\lambda_k$, and \emph{Gallagher} if Gallagher's theorem holds with $\mu$ in the place of $\lambda_k.$ 
The philosophy is that a `natural' probability measure $\mu$ should be Khinchin and Gallagher, unless there is a simple obstruction. 

A natural class of measures is given by induced Lebesgue measures on manifolds. 
The following landmark result is the culmination of decades of devoted research; see the major breakthroughs by Kleinbock and Margulis~\cite{KM1998} and by Beresnevich \cite{Ber2012}, as well as \cite{BDV2007, BVVZ2017, BVVZ2021, Ber1977, DRV1991, VV2006}.

\begin{thm}[Beresnevich--Yang \cite{BY2023}, Beresnevich--Datta \cite{BD}]
\label{DreamTheorem}
Let $\mathcal{M}$ be a non-degenerate submanifold of $\mathbb{R}^k$, where $k\geq 2.$ Let $\cU$ be an open subset of $\cM$, and let $\mu$ be a smooth probability measure supported on $\cU \cap [0,1]^k$. Then $\mu$ is Khinchin.
\end{thm}

For the Gallagher property, much less is known. Badziahin--Levesley~\cite{BL2007} solved the convergence theory for planar curves, subject to curvature. For affine hyperplanes, see \cite{BHV2020, Cho2018, CT2020, CT2024, CY2024, Hua2024}, and finally \cite{CY}, where the problem was largely solved when $k \ge 9$.
For our purposes, the most relevant result is as follows.

\begin{thm}
[Chow--Yu \cite{CY}]
Let $\mathcal{M}$ be a hypersurface in $\mathbb{R}^k$, where $k\geq 5.$ Let $\cU$ be an open subset of $\cM$, in which the Gaussian curvature is non-zero, and let $\mu$ be a smooth probability measure supported on $\cU \cap [0,1]^k$. 
Then $\mu$ is Gallagher.
\end{thm}

In this article, we investigate \emph{missing-digit measures}. These generalise the canonical measure on the middle-third Cantor set, and we will define them in \S \ref{sec: Cantor}. A motivating goal is as follows.

\begin{conj}
[Dream Theorem for missing-digit fractals]
\label{conj: main}
Let $k$ be a positive integer, and let $\mu$ be a missing-digit measure on $\mathbb{R}^k$ that is not supported on any affine hyperplane. Then $\mu$ is Khinchin and Gallagher.
\end{conj}

\begin{rem} The so-called Dream Theorem was introduced in a survey article \cite{BRV2020}, referring to Theorem \ref{DreamTheorem} before it was proved. Conjecture \ref{conj: main} is analogous, but also includes the Gallagher property. We would expect it to extend to self-similar measures \cite{Mat2015}.
\end{rem}

A measure on $\mathbb{R}^k$ is \emph{split} if it is a Cartesian product of Borel probability measures on $\mathbb{R}$. Our main result is as follows.

\begin{thm}
[Main theorem]
\label{thm: main}
Let $\mu$ be a Borel probability measure on $\mathbb{R}^k.$ If 
\begin{equation}
\label{MainAssumption}
\dim_{\ell^1}(\mu) > k - \frac{k-1}{k+1},
\end{equation}
then $\mu$ is Khinchin. If moreover
$
\mu = \mu_1 \times \cdots \times \mu_k
$
is split, and
\begin{equation}
\label{SplitAssumption}
\dim_{\ell^1}(\mu_j) > 1 - \frac1{k+1} \qquad (1 \le j \le k),
\end{equation}
then it is Gallagher.
\end{thm}

The Fourier $\ell^1$ dimension $\dim_{\ell^1}$ was used in \cite{Yu} to establish a Jarn\'ik--Besicovitch-type theorem for fractal measures; we will define it in \S \ref{sec: dimension}. Computing Fourier $\ell^1$ dimension for missing-digit measures is not a simple task. In fact, it is not even immediately clear why there are proper missing-digit measures that satisfy \eqref{MainAssumption}. For this reason, we will in \S \ref{sec: special} provide concrete examples where Theorem \ref{thm: main} applies. In particular, we will see that if the `digit structure' of a missing-digit measure is simple enough, then its Fourier $\ell^1$ and Hausdorff dimensions differ only slightly. Fourier $\ell^1$ dimension was used in \cite{YuRadial, YuManifold} to deduce various results on the continuity of projections of missing-digit measures in $\mathbb{R}^k$, where $k \ge 2$.

Recently, in \cite{CVY}, a precise method to compute $\dim_{\ell^1}(\mu)$ for an arbitrary missing-digit measure $\mu$ on $[0,1]$ was presented. We discuss a higher-dimensional analogue in \S \ref{sec: compute}.

\subsection{Related work}

Similar conjectures can be found in the articles of Levesley--Salp--Velani~\cite{LSV2007} and Bugeaud--Durand~\cite{BD2016}, and the problems are also discussed in the articles of Mahler \cite{Mahler84} and Kleinbock--Lindenstrauss--Weiss~\cite{KLW}. Recent progress towards Conjecture \ref{conj: main} can be found in \cite{CVY, DJ, EFS2011, KL2020, SW2019, Yu}. We highlight the following breakthrough.

\begin{thm}
[Khalil--Luethi \cite{KL2020}]
\label{KLthm}
Let $k \in \bN$, and let $\mu$ be a missing-digit measure on $[0,1]^k$. Then there is an effectively-computable constant $\epsilon > 0$ depending only on $k$ such that if $\Haus(\mu) > k -\epsilon$ then $\mu$ is Khinchin. 
\end{thm}

Khalil and Luethi's method applies not only to missing-digit measures but also to more general self-similar measures with rational parameters and the open set condition. The Hausdorff dimension $\Haus(\mu)$ of a Borel measure $\mu$ is defined, for example, in \cite[Chapter 2]{Fal97}, but we do not require the notion in full generality. If $\mu$ is the Cantor--Lebesgue measure of a missing-digit fractal $\cK$, see \S \ref{sec: Cantor}, then
\[
\dim_{\ell^1}(\mu) \le \Haus(\mu) = \Haus(\cK).
\]
In this setting, we will see in \S \ref{sec: special} that the difference between $\dim_{\ell^1}(\mu)$ and $\Haus(\mu) = \Haus(\cK)$ is often small, in which case Theorem \ref{thm: main} goes beyond Theorem \ref{KLthm}.

In the special case $k=1$, there has been even more dramatic recent progress.

\begin{thm}
[B\'enard--He--Zhang \cite{BHZ}]
\label{BHZthm}
Let $\mu$ be a self-similar probability measure on $[0,1]$ whose support is not a singleton. Then $\mu$ is Khinchin.
\end{thm}

Theorems \ref{KLthm} and \ref{BHZthm} both use random walks on homogeneous spaces. Our Fourier-analytic methodology, here and in \cite{CY, Yu}, is very different.

\subsection{More refined results}

In the course of proving Theorem \ref{thm: main}, we establish some finer results. In order to state them, we introduce some more refined notions. For $\psi: \bN \to [0,1)$ and $\by =(y_1,\dots,y_k)\in\mathbb{R}^k$, we denote by $W_k(\psi,\by)$ the set of  $\bx \in [0,1)^k$ such that
\[
\max\{\| n x_1-y_1 \|, \dots ,\| n x_k-y_k \|\} < \psi(n)
\] 
has infinitely many solutions $n \in \bN$, and write $W_k^\times(\psi,\by)$ for the set of $\bx \in [0,1)^k$ such that
\[
\| n x_1 - y_1 \| \cdots \| n x_k-y_k \| < \psi(n)
\]
holds for infinitely many $n \in \bN$.

\begin{defn}
Let $k \in \bN$, and let $\psi:\mathbb{N} \to [0,1)$ be non-increasing. A probability measure $\mu$ on $[0,1]^k$ is 
\begin{itemize}
\item \emph{inhomogeneous convergent Khinchin} if \[
\sum_{n=1}^{\infty} \psi(n)^k < \infty \implies \mu(W_k(\psi,\by)) = 0 \quad (\by \in \bR^k),
\]
\item \emph{inhomogeneous divergent Khinchin} if
\[
\sum_{n=1}^{\infty} \psi(n)^k = \infty \implies \mu(W_k(\psi,\by)) = 1 \quad (\by \in \bR^k),
\]
\item \emph{inhomogeneous convergent Gallagher} if
\[
\sum_{n=1}^{\infty} \psi(n) (\log n)^{k-1} < \infty \implies \mu(W^\times_k(\psi,\by)) = 0 \quad (\by \in \bR^k),
\]
\item \emph{inhomogeneous divergent Gallagher} if
\[
\sum_{n=1}^{\infty} \psi(n) (\log n)^{k-1} = \infty \implies \mu(W^\times_k(\psi,\by)) = 1 \quad (\by \in \bR^k).
\]
\end{itemize}
\end{defn}

We establish the following new results. Theorem \ref{thm: main} follows from these by specialising $\by = \bzero$ in the conclusions.

\begin{thm}
[Convergence theory]
\label{thm: convergence}
Let $k \in \bN$, and let $\mu$ be a Borel probability measure on $[0,1]^k$ satisfying
\begin{equation}
\label{WeakAssumption}
\dim_{\ell^1}(\mu) > k - \frac{k}{k+1}.
\end{equation}
Then $\mu$ is inhomogeneous convergent Khinchin. If moreover
$
\mu = \mu_1 \times \cdots \times \mu_k
$
is split and satisfies \eqref{SplitAssumption},
then it is inhomogeneous convergent Gallagher.
\end{thm}

\begin{thm}
[Divergence theory]
\label{thm: divergence}
Let $k \in \bN$, and let $\mu$ be a Borel probability measure on $[0,1]^k$ satisfying \eqref{MainAssumption}. Then $\mu$ is inhomogeneous divergent Khinchin and inhomogeneous divergent Gallagher.
\end{thm}

Observe that, for the divergent Gallagher property and for the convergent/divergent Khinchin properties, we do not require $\mu$ to be split. We only require $\mu$ to be split for the convergent Gallagher result. 
Note also that if $k = 1$ then the assumption \eqref{MainAssumption} cannot be met. We will see in \S \ref{sec: special} that if $k \ge 2$ then \eqref{MainAssumption} can hold.

\subsection{Methods}

Observe that sets like $W_k^\times(\psi)$ are limit superior sets. To show that they have full measure, we use the divergence Borel--Cantelli lemma \cite{BV2023}. This requires us to estimate sums of measures of arcs around rationals, and of the intersections of these arcs. We achieve this using Fourier analysis, specifically via the framework of `moment transference principles' that we recently developed in \cite{CY}. Unlike in the case of non-degenerate hypersurfaces \cite{CY}, the Fourier coefficients $\hat \mu(\bk)$ do not decay pointwise: \mbox{$\hat \mu(\bk) \not \to 0$} as $\| \bk \|_\infty \to \infty$. We instead exploit their average decay using the Fourier $\ell^1$ dimension --- developed in \cite{Yu} and \cite{CVY} --- as discussed in \S \ref{sec: DAfrac}. Divisibility also plays a key role, since the $n^{-1}$-periodicity of the arcs manifests as $n$-divisibility in Fourier space $\bZ^k$. To establish the variance transference principle, we ultimately count solutions to a system of Diophantine equations.

\subsection{A bound on the Hausdorff dimension}

Our expectation transference principle is strong enough to establish a Hausdorff dimension bound for simultaneous approximation on missing-digit fractals. To put this into context, we recall Levesley's inhomogeneous version of the classical Jarn\'ik--Besicovitch theorem. For $t \ge 0$ and $n \in \bN$, let $\psi_t(n) = n^{-t}$.

\begin{thm}
[Inhomogeneous Jarn\'ik--Besicovitch \cite{Lev1998}]
\label{LevesleyThm}
Let $k \in \bN$ and $\bgam \in \bR$. Then
\[
\Haus(W_k(\psi_t,\by)) = \min \left\{ \frac{k+1}{t+1}, k \right\}.
\]
\end{thm}

Our result is as follows.

\begin{thm}
\label{HausdorffConvThm}
Let $k \in \bN$ and $\by \in \bR^k$.
Then there exists $\eps > 0$ such that the following holds whenever
\[
t < \frac1k + \eps.
\]
Let $\cK$ be a missing-digit fractal in $[0,1]^k$ whose Cantor--Lebesgue measure $\mu$ satisfies 
\eqref{WeakAssumption}.
Then
\[
\Haus(W_{k}(\psi_t,\by) \cap \cK) \leq \Haus(W_k(\psi_t,\by)) + \Haus(\cK) - k.
\]
\end{thm}

\begin{rem}
This upper bound was first conjectured with equality in \cite{BD2016}, when $\cK$ is the middle-third Cantor set and $\by = \bzero$. Despite the recent progress made in \cite{Yu} and \cite{CVY}, the problem remains open. The conjectured threshold differs for larger values of $t$, transitioning to $\Haus(\cK)/(t + 1)$.
\end{rem}

\subsection{Counting rational points on/near missing-digit fractals}

For any compact set $\cK \subset \mathbb{R}^k$, and real numbers $Q \ge 1$ and $\del \ge 0$, we define
\[
\cN_\cK(Q,\delta) =
\{ (\ba,q) \in \bZ^k \times (\bZ \cap [1,Q]): \dist(\ba/q, \cK) \le \delta/Q \}.
\]
Observe that $\mathcal{N}_\cK(Q,0)$ comprises rational points in $\cK$ of height at most $Q$. A challenging problem is to show that if $\cK$ `lacks linear structure' but has a naturally-associated dimension $\dim(\cK)$ then, for any $\varepsilon>0,$
\[
\#\mathcal{N}_\cK(Q,0) \ll_\eps Q^{\dim (\cK) + \epsilon}.
\]
For missing-digit fractals, we have $\dim (\cK) = \Haus (\cK).$ 

For irreducible projective varieties, the problem is known as Serre's dimension growth conjecture \cite{Bro2009, Serre}. It was solved by Salberger in 2009, and that solution was recently published \cite{Sal2023}.

For the middle-third Cantor set, the conjecture was made by Broderick, Fishman, and Reich \cite{BFR11}. The statement naturally extends to missing-digit fractals. In the case $k=1$, progress was made recently in \cite{CVY}, but not until now has substantial progress been made for $k \ge 2$.

For $\del \in (0,1)$, heuristic reasoning suggests that
\begin{equation}
\label{HeuristicEstimate}
\#\mathcal{N}_\cK(Q,\delta) \asymp \delta^{k-\dim (\cK)} Q^{\dim(\cK)+1}.
\end{equation}
The critical threshold, corresponding to Dirichlet's approximation theorem, is $\del = Q^{-1/k}$. The problem becomes significantly more demanding when $\delta$ is smaller than this, and the estimate can even break down when $\delta$ is extremely small. 
A related problem was introduced by Mazur~\cite{Mazur}, who asked how close a rational point can be to a smooth planar curve and still miss.

When $\cK$ is a suitably curved manifold, we now have some good counting estimates with $\delta$ below the $Q^{-1/k}$ threshold, especially for hypersurfaces, see for instance \cite{BVVZ2021, Hua2020, SchYam2022, S24}.
For certain missing-digit fractals on $[0,1]$, such a result was obtained in \cite{CVY}. Our present methods deliver \eqref{HeuristicEstimate} beyond the critical threshold for missing-digit fractals on $[0,1]^k$.

\begin{thm}
\label{thm: Rational Counting}
Let $k \in \bN$, and let $\cK$ be a proper missing-digit fractal in $[0,1]^k$ with Cantor--Lebesgue measure $\mu$ satisfying \eqref{WeakAssumption}. Then, for some constant $\eta > 0,$ we have \eqref{HeuristicEstimate}
whenever $Q$ is sufficiently large and $\del \gg Q^{-\eta-1/k}$.
\end{thm}

We can apply Theorem \ref{thm: Rational Counting} to count rational points on missing-digit fractals.

\begin{cor}
\label{CountingRationalPoints}
Let $k \in \bN$, and let $\cK$ be a proper missing-digit fractal in $[0,1]^k$ with Cantor--Lebesgue measure $\mu$ satisfying \eqref{WeakAssumption}. Then, for some constant $\eta > 0,$ 
\[
\# \cN_{\cK}(Q,0) \ll
Q^{(1+1/k)\Haus(\cK) - \eta}.
\]
\end{cor}

\begin{proof} By Theorem \ref{thm: Rational Counting} with $\eta/(k-\Haus(\cK))$ in place of $\eta$,
\begin{align*}
\# \cN_{\cK}(Q,0) &\le \# \cN_{\cK}(Q,Q^{-\eta-1/k})  \\ &\ll
Q^{-\eta} (Q^{-1/k})^{k-\Haus(\cK)} Q^{\Haus(\cK) + 1} = Q^{(1+1/k)\Haus(\cK) -\eta}.
\end{align*}
\end{proof}

This generalises the $k=1$ case, which was demonstrated in \cite{CVY}. Counting rationals in the middle-third Cantor set is considered to be a difficult problem, see for instance \cite{RSTW2020}. As explained in \cite{CVY}, the condition \eqref{WeakAssumption} fails for the middle-third Cantor set. The condition holds, for instance, when 
\[
k = 1, \qquad b \ge 5, \qquad \# \cD = b-1.
\]
See \S \ref{sec: special} for some situations in which $k \in \bN$ is arbitrary the condition holds.

The assumption \eqref{WeakAssumption} is essential for the proof of Theorem \ref{thm: Rational Counting}, but we suspect that the weaker assumption $\Haus(\mu) > k-\frac{k}{k+1}$ should suffice for the result. It is easy to see that such a condition cannot be dropped without imposing other conditions, e.g. that $\mu$ is not supported in any proper affine subspace. We illustrate with the following example, using the notation of \S \ref{sec: prep}.

\begin{ex} \label{ex: sharp counting}
Suppose $\cK = \cK_{b,\cD}$ is a missing-digit fractal in $[0,1]^k$ with
\[
\cD \supseteq \{ 0, 1, \ldots, b-1 \}^{k-1} \times \{ 0 \}.
\]
As $\cK$ contains $[0,1]^{k-1} \times \{ 0 \}$, we have
\[
\# \cN_{\cK}(Q,0) \gg Q^k.
\]
By \eqref{HausdorffMissing}, such a set $\cK$ exists with
\[
\Haus(\cK) < k - \frac{k}{k+1}.
\]
If moreover $\del \le Q^{-1/k}$, then
\[
\delta^{k - \Haus(\cK)} Q^{\Haus(\cK) + 1} \le Q^{(k+1)\Haus(\cK)/k} = o(Q^k).
\]
\end{ex}

The constant $\eta$ in Theorem \ref{thm: Rational Counting} can be chosen according to a lower bound for $\dim_{\ell^1}(\mu)$, which can be effectively computed using Theorem~\ref{l1lower}. In special cases, we obtain a bound for $\# \cN_\cK(Q,0)$ with an exponent that is arbitrarily close to best possible.

\begin{thm} \label{OtherBound}
Fix $\eta > 0.$ Then there exists a missing-digit fractal $\cK \supseteq [0,1]^{k-1} \times \{ 0 \}$ such that
\[
\# \cN_{\cK}(Q,0) \ll Q^{k+\eta},
\qquad \Haus(\cK) > k - \eta.
\]
\end{thm}

Note that if $\Haus(\cK) > k -\frac{k}{k+1}$ then the estimate 
\[
\# \cN_{\cK}(Q,0) \ll Q^{k+\eta}
\]
does not follow trivially from Corollary \ref{CountingRationalPoints}.
Since $\cK \supseteq [0,1]^{k-1} \times \{ 0 \}$, the bound is sharp up to a factor of $O(Q^\eta)$. However, we suspect that better estimates hold for missing-digit fractals that do not concentrate too much on affine subspaces.

\subsection{Intrinsic Diophantine approximation}

Although the main focus of this paper is extrinsic Diophantine approximation, we also establish new results in the intrinsic setting, wherein we approximate by rational vectors in $\cK$. The terminology was introduced by Mahler in the influential article \cite{Mahler84}. Let $\cK \subsetneq [0,1]^k$ be a missing-digit fractal with Cantor--Lebesgue measure $\mu$, and let $\psi:\mathbb{N} \to [0,1)$. We denote by $W^\cK_k(\psi)$ the set of $\bx = (x_1,\dots,x_k) \in [0,1)^k$ such that
\[
\max\left\{\left|x_1-\frac{a_1}{n}\right|,\dots,\left|x_k-\frac{a_k}{n}\right|\right\}<\frac{\psi(n)}{n}
\]
holds for infinitely many $(a_1,\dots,a_k,n) \in\mathbb{Z}^k \times \mathbb{N}$ for which
\[
\left(\frac{a_1}{n},\dots,\frac{a_k}{n}\right) \in \cK.
\]

For $\tau>0,$ we write $W_k^\cK(\tau) = W_k^K(n \mapsto n^{-\tau}).$ It follows from a result of Weiss \cite{W01} that
\begin{equation} \label{WeissCor}
\mu(W_1^\cK(\tau))=0 \qquad (\tau > 1).
\end{equation}
More generally, it follows from a result of Kleinbock--Lindenstrauss--Weiss \cite{KLW} that if $\mu$ satisfies an affine irreducibility condition then
\begin{equation} \label{KLWcor}
\mu(W_k^\cK(\tau)) = 0
\qquad (\tau > 1/k).
\end{equation}

\begin{conj}
[Broderick--Fishman--Reich \cite{BFR11}] \label{IntrinsicVWAconjecture} Let $k \in \bN$, and let $\cK$ be a proper missing-digit fractal in $[0,1]^k$ that is not contained in any proper affine subspace. Then, for all $\tau>0,$
\[
\mu(W_k^\cK(\tau)) = 0.
\]
\end{conj}

This conjecture was originally posed for the middle-third Cantor set, but we believe that it should hold in general. A recent result in \cite{CVY} asserts that if $\dim_{\ell^1}(\mu) > 1/2$ then there exists $\eta > 0$ such that 
\begin{equation}
\label{CVYintrinsic}
\mu(W_1^\cK(\tau))=0
\qquad (\tau > 1 - \eta).
\end{equation}
This refined \eqref{WeissCor}. Presently, we extend \eqref{CVYintrinsic} to higher dimensions, thereby refining \eqref{KLWcor}.

\begin{thm}
\label{thm: Intrinsic App}
Let $k \in \bN$, and let $\cK$ be a proper missing-digit fractal in $[0,1]^k$ with Cantor--Lebesgue measure $\mu$ satisfying \eqref{WeakAssumption}.
Then, for some $\eps > 0$, 
\[
\mu(W_k^\cK(\tau)) = 0
\qquad \left( \tau > \frac1k - \eps \right).
\]
\end{thm}

We will see that this is a simple consequence of Corollary \ref{CountingRationalPoints}. Note that if \eqref{WeakAssumption} holds then $\mu$ cannot be supported in any proper affine subspace.

\subsection{Further related work}

One can simplify the problems of counting rationals and intrinsic approximation by using the intrinsic height \cite{FS2014} instead of the denominator. For the middle-third Cantor set, see the recent article \cite{TWW2024}.

Mahler asked if there are very well approximable numbers in the middle-third Cantor set that are irrational and non-Liouville. To solve this problem, Levesley--Salp--Velani \cite{LSV2007} resolved the metric theory of approximation by triadic rationals in the middle-third Cantor set.

Velani asked about the analogous problem for dyadic approximation. This is related to Furstenburg's `times two, times three' problem \cite{Fur1967}, and was first investigated in \cite{ACY2023}, with subsequent developments in \cite{Bak2024} and most recently \cite{ABCY2023}. Velani's conjecture is still very much open.

\subsection*{Organisation}

In \S \ref{sec: prep}, we gather some preliminaries in describe our general setup. In \S \ref{sec: ETP}, we establish the ETP. In \S \ref{sec: VTP}, we establish the VTP. In \S \ref{sec: final proofs}, we complete the proofs of Theorems \ref{thm: convergence}, \ref{thm: divergence}, \ref{HausdorffConvThm}, \ref{thm: Rational Counting}, \ref{OtherBound}, and \ref{thm: Intrinsic App}. In Appendix \ref{sec: morel1}, we further discuss Fourier $\ell^1$ dimension for missing-digit measures.

\subsection*{Notation}

For complex-valued functions $f$ and $g$, we write $f \ll g$ or $f = O(g)$ if there exists $C$ such that $|f| \le C|g|$ pointwise, we write $f \sim g$ if $f/g \to 1$ in some specified limit, and we write $f = o(g)$ if $f/g \to 0$ in some specified limit. We will work in $k$-dimensional Euclidean space, and the implied constants will always be allowed to depend on $k$, as well as on any parameters which we describe as fixed or constant. Any further dependence will be indicated with a subscript.

For $x \in \bR$, we write $e(x) = e^{2 \pi i x}$. For $r > 0$ and $\by \in \bR^k$, we write $B_r(\by)$ for the ball of radius $r$ centred at $\by$, and put $B_r = B_r(\bzero)$.

\subsection*{Funding}

HY was supported by the Leverhulme Trust (ECF-2023-186).

\subsection*{Rights}

For the purpose of open access, the authors have applied a Creative Commons Attribution (CC-BY) licence to any Author Accepted Manuscript version arising from this submission.

\section{Preparation}
\label{sec: prep}

In this section, we gather some preliminaries and describe our general setup.

\subsection{Missing-digit fractals and measures}
\label{sec: Cantor}

Let $k \ge 1$ and $b \ge 2$ be integers, and let $\bP$ be a probability measure on 
$\{ 0,1,\ldots,b-1 \}^k$.
A \emph{missing-digit measure} is the probability measure $\mu_{b,\bP}$ whose distribution is that of the random variable
\[
\sum_{j=1}^\infty \bd^{(j)}/b^j,
\]
where the $\bd^{(j)} \in 
\{0,1,\ldots,b-1\}^k$ are i.i.d. with distribution $\bP$. A \emph{missing-digit fractal} is the support of a missing-digit measure. Equivalently, a missing-digit fractal is
\[
\cK_{b,\cD} =
\left \{ \sum_{j=1}^\infty \bd^{(j)}/b^j: \bd^{(j)} \in \cD \right \} 
\subseteq [0,1]^k,
\]
where $k,b$ are as above and $\cD \subseteq \{ 0, 1, \ldots, b-1 \}^k$.
The \emph{Cantor--Lebesgue measure} of a missing-digit fractal $\cK_{b,\cD}$ is 
\[
\mu_{b, \cD} = \mu_{b,\bP}, \qquad
\text{where} \qquad
\bP(\bd) = \# \cD^{-1}
\quad (d \in \cD).
\]

\begin{ex} The Cantor set is 
$\cK_{b, \{0,2\}}$. 
The Cantor measure is the Cantor--Lebesgue measure of the Cantor set.
\end{ex}

A missing-digit fractal $\cK$ is \emph{proper} if $\# \cK \ge 2$ and $\cK \ne [0,1]^k$. A missing-digit measure is \emph{proper} if its support is proper.

\subsection{Notions of dimension}
\label{sec: dimension}

We begin our discussion with Hausdorff measure and dimension, which are standard concepts in fractal geometry \cite{Fal2014}. For $\cU \subseteq \bR^k$, we write $|\cU|$ for the diameter of $\cU$. 

Let $k \in \bN$ and $\cF \subseteq \bR^k$. For $s \ge 0$ and $\del > 0$, define
\[
H^s_\del(\cF) = \inf \left \{
\sum_{i=1}^\infty |\cU_i|^s: \cF \subseteq \displaystyle \bigcup_{i=1}^\infty \cU_i, \quad |\cU_i| \le \del \: \forall i \right \}.
\]
The \emph{Hausdorff $s$-measure} of $\cF$ is
\[
H^s(\cF) = \lim_{\del \to 0} H^s_\del(\cF),
\]
and the \emph{Hausdorff dimension} of $\cF$ is
\[
\Haus(\cF) = \inf \{ s \ge 0: H^s(\cF) = 0 \}.
\]
Hausdorff measure refines Lebesgue measure, in that $H^k(\cF) = c_k \lam_k(\cF)$ for some constant $c_k > 0$. Note that
\begin{equation}
\label{HausdorffMissing}
\Haus(\cK_{b,\cD}) = \frac{\log \# \cD}{\log b}.
\end{equation}

We equip ourselves the following basic tool, which is commonly used to establish upper bounds for Hausdorff measure and dimension.

\begin{lem} 
[Hausdorff--Cantelli lemma]
\label{HausdorffCantelli}
Let $k \in \bN$, let $\cE_1, \cE_2, \ldots$ be hypercubes in $\bR^k$, and put $\cE_\infty = \displaystyle \limsup_{n \to \infty} \cE_n$. Let $s > 0$, and suppose
\[
\sum_{j=1}^\infty |\cE_j|^s < \infty.
\]
Then $H^s(\cE_\infty) = 0$ and $\Haus(\cE_\infty) \le s$.
\end{lem}

\begin{proof} This is \cite[Lemma 3.10]{BD1999}.
\end{proof}

\bigskip

Next, we recall the Fourier $\ell^t$ dimensions, which were brought into this subject in \cite{Yu}. Let $\mu$ be a Borel probability measure supported on $[0,1]^k$. For $\bxi \in \bZ^k$, the $\bxi^{\mathrm{th}}$ \emph{Fourier coefficient} of $\mu$ is
\[
\hat \mu(\bxi) = \int_{[0,1]^k} e(-\bxi \cdot \bx) \d \mu(\bx).
\]
The \emph{Fourier $\ell^t$ dimension} of $\mu$ is
\[
\dim_{\ell^t}(\mu) = \sup \left \{ s \ge 0:
\sum_{\| \bxi \|_\infty \le Q} |\hat \mu(\bxi)|^t \ll Q^{k - s} \right \}.
\]
It follows from the Cauchy--Schwarz inequality that
\[
\frac{\dim_{\ell^2}(\mu)}{2} \le \dim_{\ell^1}(\mu) \le \dim_{\ell^2}(\mu).
\]

For $s > 0$, we say that $\mu$ is \emph{$s$-regular} if
\[
\mu(B_r(\by)) \asymp r^s \qquad
(\by \in \supp(\mu), \quad 0 < r \le 1).
\]
Note that the Cantor--Lebesgue measure of a missing-digit fractal $\cK$ is $s$-regular with $s = \dimh(\cK)$, see \cite[Theorem 3.9]{Far2015}. By mimicking the continuous analogue in \cite[\S 3.8]{Mat2015}, it is readily confirmed that if $\mu$ is $s$-regular then
\[
\dim_{\ell^2}(\mu) = \Haus(\supp(\mu)) = s.
\]
Note also that if $\mu$ is a missing-digit measure and $\Haus(\supp(\mu)) = s$ then
\begin{equation}
\label{MeasureUpper}
\mu(B_r(\by)) \ll r^s \qquad
(\by \in \bR^k, \quad 0 < r \le 1).
\end{equation}

\subsection{Moment transference principles}
\label{sec: MTP}

For any Borel probability measure $\mu$ on  $[0,1]^k$ and any $f \in L^1(\mu)$, we write
\[
\mu(f) = \int_{[0,1]^k} f \d \mu.
\]
Let $\lam$ be Lebesgue measure on $[0,1]^k$. Let $\cE_1, \cE_2, \ldots$ be bounded Borel subsets of $\bR^k$, and put $\cE_\infty = \displaystyle \limsup_{n \to \infty} \cE_n$. For $n \in \bN$, let $f_n: \bR^k \to [0,\infty)$ be compactly supported and measurable.

Define
\begin{align*}
E_N(\mu) = \sum_{n \le N} \mu(f_n), \qquad
V_N(\mu) = \int_{[0,1]^k} \left( \sum_{n \le N} (f_n(\bx) - \lam(f_n)) \right)^2 \d \mu(\bx).
\end{align*}
The \emph{expectation transference principle} (ETP) holds if
\[
E_N(\mu) = (1 + o(1)) E_N(\lam) + O(1) \qquad (N \to \infty).
\]
The \emph{variance transference principle} (VTP) holds if
\[
V_N(\mu) = V_N(\lam) + o(E_N(\mu)^2) + O(1),
\]
along a sequence of $N \to \infty$.

\begin{lem} [General convergence theory conclusion]
\label{GenConv}
Let $\psi: \bN \to [0,1)$. Assume that $f_n \gg 1$ on $\cE_n$ for all sufficiently large $n \in \bN$, that
\begin{equation}
\label{LebesgueConvergence}
\sum_{n=1}^\infty \lam(f_n) < \infty,
\end{equation}
and that we have ETP for this sequence of functions. Then $\mu(\cE_\infty) = 0$.
\end{lem}

\begin{proof} Replace $A_n^\times$ by $\cE_n$ in the proof of \cite[Lemma 2.6]{CY}.
\end{proof}

\begin{lem} [General divergence theory conclusion]
\label{GenDiv}
Let $C \ge 1$ and $\psi: \bN \to [0,1)$. For $n \in \bN$, suppose $f_n$ is supported on $\cE_n$. Suppose we have ETP and VTP for this sequence of functions, as well as
\begin{equation}
\label{LebesgueDivergence}
\sum_{n=1}^\infty \lam(f_n) = \infty
\end{equation}
and
\begin{equation}
\label{LebesgueQIA}
\sum_{m,n \le N} \lam(f_m f_n) \le
C E_N(\lam)^2 + O(1).
\end{equation}
Then $\mu(\cE_\infty) \ge 1/C$.
\end{lem}

\begin{proof} The ETP and the divergence of $\sum_n \lam(f_n)$ give
\[
\sum_{n=1}^\infty \mu(f_n) = \infty.
\]
By \cite[Lemma 2.5]{CY},
\[
\sum_{m,n \le N} \mu(f_m f_n) = \sum_{m,n \le N} \lam(f_m f_n) + o(E_N(\mu)^2) + O(1),
\]
along a sequence of $N \to \infty$. Combining this with \eqref{LebesgueQIA} and the ETP yields
\[
\sum_{m,n \le N} \mu(f_m f_n) \le (C + o(1)) E_N(\mu)^2 + O(1)
\]
along this sequence. Functional divergence Borel--Cantelli \cite[Lemma 2.7]{CY} completes the proof.
\end{proof}

\subsection{Setup}
\label{sec: rectangle}

Our analysis involves decomposing the problem into rectangles, scaling bump functions accordingly, and passing to Fourier space. Let $w: \bR \to [0,1]$ be a non-zero bump function supported on $[-2,2]$. This has the property that
\begin{equation}
\label{decay}
\hat w(\xi) \ll_L 
|\xi|^{-L},
\end{equation}
for any $L \ge 1$. The function $w$ will approximate $\{ x: |x| \le 1\}$ or $\{ x: 1/2 \le |x| \le 1 \}$.

Next, we adapt $w$ to certain rectangles. Let $\by \in \bR^k$ and $\bd \in (0,1/2]^k$, and define
\begin{align*}
A_n(\bd) &=
\{ \bx \in [0,1]^k:
\| n x_j - y_j \| < n d_j \quad (1 \le j \le k) \}, \\
A^\circ_n(\bd) &=
\{ \bx \in [0,1]^k:
n d_j/2 < \| n x_j - y_j \| < n d_j \quad (1 \le j \le k) \},\\
\cR(\bd) &= [-d_1, d_1] \times
\cdots \times [-d_k, d_k], \\
\cR^\vee(\bd) &= [-1/d_1,1/d_1] \times \cdots \times [-1/d_k, 1/d_k], \\
w_{\cR(\bd)} (\bx) &= \prod_{j \le k} w(x_j/d_j).
\end{align*}
Then
\[
\widehat{w_{\cR(\bd)}}(\bxi) = \prod_{j \le k} d_j \hat w(d_j \xi_j).
\]
Note that $w_{\cR(\bd)}$ is normalised so that
\[
\widehat{w_{\cR(\bd)}}(\bzero)\asymp d_1\cdots d_k.
\]
If $\bd$ is clear from the context, then we will simply write $\cR$ for $\cR(\bd)$ and $\cR^\vee$ for $\cR^\vee(\bd)$. By \eqref{decay}, we thus have
\begin{equation}
\label{decay2}
\frac{\widehat{w_\cR}(\bxi)}
{d_1 \cdots d_k} \ll 
2^{-mL}
\qquad
(\bxi \in 2^m \cR^\vee \setminus 2^{m-1} \cR^\vee).
\end{equation}

We see that $A_n(\bd)$ and $\cR$ have the same shape, and note that $\cR^\vee$ is the dual rectangle of $\cR.$ The bump function $w_{\cR}$ is supported on 
\[
2 \cR := \{ 2 \bx: \bx \in \cR \}.
\]
Its Fourier transform $\hat{w}_{\cR}$ is essentially supported on $\cR^\vee.$ For any shift $\bgam\in\mathbb{R}^k,$ the shifted function $w_{\cR}(\cdot-\bgam)$ is supported on $\cR+\bgam$, and its Fourier transform is still essentially supported on $\cR^\vee.$ These considerations can be used to establish ETP and VTP for unions of rectangles. 

We will work with a smooth, periodically extended version of the indicator function of $A_n(\bd)$, namely
\begin{equation}
\label{AnStar}
A^*_n(\bx;\bd) = \sum_{\ba \in  \mathbb{Z}^k} w_{\cR(\bd)} \left( \bx - \frac{\ba+\by}{n} \right),
\end{equation}
where $w$ approximates $\{ x: |x| \le 1 \}$. By choosing $w$ to instead approximate $\{ x: 1/2 \le |x| \le 1 \}$, we can alternatively use this same expression to approximate $A_n^\circ(\bd)$.

For the simultaneous theory, we use $A_n(\bd)$ with $\bd =(d,d,\dots,d)$ for some $d>0.$ Thus, the corresponding $\cR$ and $\cR^\vee$ are all hypercubes. 

For the multiplicative theory, the situation is more complicated, as we need to consider boxes of different shapes in order to approximate 
\[
A_n^\times := \{ \bx \in [0,1)^k:
\| n x_1 - y_1 \| \cdots
\| n x_k - y_k \| < \psi(n) \}.
\]
We write
\[
h_i = 2^{-i}
\]
for each integer $i$ in some range. The range is given by
\[
\frac{\psi(2^{m-1})}{2^{m-1}} \le h_i \le \frac1{2^{m-1}}
\]
in the convergence setting, for
\[
n \in D_m := [2^{m-1}, 2^m).
\]
Fixing a small constant $\tau > 0$, the range in the divergence setting is
\[
\frac1{2^{(1+(1+\tau)/k)m}} \ll h_i \ll \frac1{2^{(1+(1-\tau)/k)m}}.
\]
Writing
\[
h h_{i_1} \cdots h_{i_{k-1}} = \frac{\psi(2^m)}{2^{km}}, \qquad
H h_{i_1} \cdots h_{i_{k-1}} = \frac{\psi(2^{m-1})}{2^{k(m-1)}}
\]
and
\begin{align*}
\fB_n &= \bigsqcup_{h_{i_1},\ldots,
h_{i_{k-1}}}
A_n^\circ(h_{i_1}, \ldots, h_{i_{k-1}}, h), \\
\fC_n &= \bigcup_{h_{i_1},\ldots,
h_{i_{k-1}}}
A_n(h_{i_1}, \ldots, h_{i_{k-1}}, 2^{k-1}H),
\end{align*}
we then have
\[
\fB_n \subseteq A_n^\times \subseteq \fC_n.
\]

\subsection{Admissible set systems and admissible function systems}

For each $m \in \bN$, we choose an index set $I_m$ and a threshold
\[
\fd_m \gg \frac1{n^{1+(1+\tau)/k}}
\qquad (n \in D_m).
\]
For each $i \in I_m$ and each $j \in \{1,2,\ldots,k \}$, we choose $d_{i,j} = d_{i,j,m}$ such that
\[
n^{-\tau} \fd_m \ll d_{i,j} \ll n^{\tau} \fd_m  \qquad (n \in D_m),
\]
and such that $\prod_{j \le k} d_{i,j}$ does not depend on $i$. For each $n \in D_m$, we introduce a disjoint union
\[
A_n = \bigsqcup_{i \in I_m} A_n^\circ(d_{i,1},\ldots,d_{i,k}).
\]
An \emph{admissible set system} is a sequence $(A_n)_{n=1}^\infty$ obtained in this way.

For the Khinchin theory, we choose $I_m$ to be a singleton for each $m$ and $d_{i,j}$ all equal.

\begin{ex} 
[Simultaneous approximation]
\label{SimOK}
We claim that for the divergent Khinchin theory, we may assume that
\[
\psi(n) \ge \psi_L(n) :=
(n (\log n)^2)^{-1/k}
\qquad (n \ge 2).
\]
To see this, put $\tilde \psi = \max \{ \psi, \psi_L \}$ and note that
\[
W_k(\tilde \psi; \by) = W_k(\psi;\by) \cup W_k(\psi_L;\by).
\]
As $W_k(\psi_L;\by)$ has zero measure, by the convergence theory, the set $W_k(\tilde \psi;\by)$ has the same measure as $W_k(\psi;\by)$, and we may therefore replace $\psi$ by $\tilde \psi$. For $n \in D_m$, we can then use $A_n^\circ(d,\ldots,d)$ with
\[
d \asymp \frac{\psi(2^m)}{2^m}.
\]
\end{ex}

\bigskip

For the Gallagher theory, we form an admissible set system as follows. 

\begin{ex} 
[Multiplicative approximation]
\label{MultOK}
We claim that for the divergent Gallagher theory we may assume that
\[
\psi(n) \ge \psi_L(n) := \frac1{n (\log n)^{k+1}} \qquad (n \ge 2).
\]
To see this, put $\tilde \psi = \max \{ \psi, \psi_L \}$ and note that
\[
W_k^\times(\tilde \psi; \by) = W_k^\times(\psi; \by) \cup W_k^\times(\psi_L; \by).
\]
As $W_k^\times(\psi_L; \by)$ has zero measure, by the convergence theory, the set $W_k^\times(\tilde \psi; \by)$ has the same measure as $W_k^\times(\psi; \by)$, and we may therefore replace $\psi$ by $\tilde \psi$.
For $n \in D_m$, we can then pack $A_n^\times$ with roughly $m^{k-1}$ many sets $$A_n^\circ(d_{i,1},\ldots,d_{i,k})$$ as above, forming an admissible set system with
\[
\prod_{j \le k} d_{i,j} \asymp \frac{\psi(2^m)}{2^{km}}.
\]
\end{ex}

\begin{rem} \label{AssumeUpper}
For the divergent Khinchin theory, we may also assume that
\[
\psi(2^{m}) \ll 2^{-m/k} \qquad (m \in \bN).
\]
Indeed, if $\psi(2^{m}) > 2^{-m/k}$ for infinitely many $m$, then we may replace $\psi$ by a non-increasing function $\psi': \bN \to [0,1)$ such that 
\[
\psi' \le \psi, \qquad
\psi'(2^{m}) \ll 2^{-m/k} \quad (m \in \bN), \qquad
\sum_{n=1}^\infty \psi'(n)^k = \infty.
\]
To construct $\psi'$, one can define
\[
\psi'(n) = 2^{-m/k} \qquad (n \in D_m, \quad \psi(2^m) > 2^{-m/k})
\]
and interpolate between these ranges.

Similarly, for the divergent Gallagher theory, we may assume that
\[
\psi(2^{m}) \ll 2^{-m} \qquad (m \in \bN).
\]
These observations simplify matters, but our approach is robust enough to succeed even without them.
\end{rem}

\bigskip

Next, we introduce a functional analogue. Let $(A_n)_{n=1}^\infty$ be an admissible set system. For $n \in D_m$, let
\[
f_n(\bx) = \sum_{i \in I_m} A_n^*(\bd^{(i)};\bx),
\]
where
\[
\bd^{(i)} = (d_{i,1}, \ldots, d_{i,k}),
\]
and $A_n^*(\bd;\bx)$ is given by \eqref{AnStar}.
Furthermore, suppose the bump function $w$, from \S \ref{sec: rectangle}, is supported on 
\[
\{ x \in \bR: 1/2 \le |x| \le 1 \}.
\]
An \emph{admissible function system} is a sequence $(f_n)_{n=1}^\infty$ constructed in this way.

\subsection{The Lebesgue second moment}

In the Gallagher setting, this was estimated in our previous article, see \cite[Remark 6.3, Theorem~6.5, and Lemma 6.6]{CY}. 

\begin{thm} [Chow--Yu \cite{CY}]
\label{LebesgueEstimate}
Let $\psi: \bN \to [0,1)$ be non-increasing, and let $k \ge 2$ be an integer. To any non-zero bump function $\bR \to [0,1]$ supported on $[-2,2]$, we may associate a sequence of smooth functions $(f_n)_{n=1}^\infty$ in $\bR^k$ as in Example \ref{SimOK} or \ref{MultOK}. There exists such a bump function $w$ supported on $\{ x \in \bR: 1/2 \le |x| \le 1 \}$ such that, for some $C = C(k,w) > 1$ arbitrarily close to 1, we have \eqref{LebesgueQIA}. 
\end{thm}

\section{Expectation transference principle}
\label{sec: ETP}

In this section, we establish some ETP results. For notational convenience, we write
\[
\tilde E_N(\mu) = E_{2N-1}(\mu) - E_{N-1}(\mu) - E_{2N-1}(\lam) + E_{N-1}(\lam).
\]

\begin{thm} [Dyadic split ETP]
\label{thm: ETP}
Let $N \in \bN$ and $\bd \in (0,1/2]^k$. For each $n \in [N,2N)$, we associate $A_n^*$ to $\bd$ as in \S \ref{sec: rectangle}. Let $\mu = \mu_1 \times \dots \times \mu_k$ be a product of probability measures on $[0,1]$ with $\dim_{\ell^1}(\mu_j) = \kappa_j$ for each $j.$ Then, for any $\eps > 0$,
\[
\tilde E_N(\mu) \ll_\eps 
N^{k+\eps} \prod_{j=1}^{k} d_j^{\kap_j - \eps}.
\]
\end{thm}

\begin{proof}
By Parseval,
\begin{align*}
\tilde E_N(\mu) &= \sum_{N \le n < 2N} \int (A_n^*(\bx)-\lam(A_n^*))\d\mu(\bx) = \sum_{N \le n < 2N} \sum_{\bxi \neq \mathbf{0}} \widehat{A_n^*}(\bxi) \hat{\mu} (-\bxi).
\end{align*}
To proceed, the key intuition is that $\widehat{A_n^*}$ is essentially supported on the multiples of $n$ in the box
\[
\cR^\vee_N := [-1/d_1, 1/d_1] \times \cdots \times [-1/d_k, 1/d_k].
\]
Indeed, as $A_n^*$ is $n^{-1}$-periodic, its Fourier coefficients vanish away from multiples of $n$.
Moreover, we can see from \eqref{decay2} that if $u \in \bN$ and $\bxi \in 2^u \cR^\vee_N \setminus 2^{u-1} \cR^\vee_N$ then
\[
\widehat{A_n^*}(\bxi)
\ll_L 2^{-uL} N^k d_1 \cdots d_k,
\]
for any constant $L \ge 1$.
Thus,
\begin{align*}
\frac{\tilde E_N(\mu)}{N^k d_1\cdots d_k} \ll \sum_{\substack{n \mid \bxi \in \cR^\vee_N \setminus \{ \bzero \} \\ N \le n < 2N}} |\hat{\mu}(\bxi)|
+ \sum_{u=1}^\infty 2^{-uL} \sum_{\substack{n \mid \bxi \in 2^u \cR^\vee_N \setminus 2^{u-1} \cR^\vee_N 
 \\ N \le n < 2N}} 
|\hat{\mu}(\bxi)|.
\end{align*}

Next, we exploit the product structure of $\mu$. Observe that there are $O(N^\epsilon)$ many integers $n$ such that $n \mid \bxi$, whence
\begin{align*}
\sum_{\substack{n \mid \bxi \in \cR^\vee_N \setminus \{ \bzero \} \\ N \le n < 2N}} |\hat{\mu}(\bxi)| 
&\ll N^\epsilon\prod_{j \le k} \left(\frac{1}{d_j}\right)^{1-\kap_j+\epsilon}.
\end{align*}
For $u \in \bN$,
\begin{align*}
\sum_{\substack{n\mid \bxi \in 2^u \cR^\vee_N\setminus 2^{u-1} \cR^\vee_N \\ N \le n < 2N}} |\hat{\mu}(\bxi)| &\ll (2^u N)^\epsilon \prod_{j \le k} \left(\frac{2^u}{d_j}\right)^{1-\kap_j + \epsilon} \\
&\le 2^{Ku} N^{\epsilon} 
\prod_{j \le k}
\left(\frac{1}{d_j}
\right)^{1 - \kap_j + \eps},
\end{align*}
where $K = (k+1)(1+\eps)$. Choosing $L > K$ then gives
\begin{align*}
\label{eqn: last step}
\frac{\tilde E_N(\mu)}{N^k d_1\cdots d_k} \ll N^{\epsilon} \prod_{j=1}^{k} 
\left(\frac{1}{d_j}
\right)^{1 - \kap_j + \eps}.
\end{align*}
\end{proof}

We also need an ETP result for general measures, rather than product measures. For this, we require a strong assumption on $\bd.$

\begin{thm}
[Dyadic general ETP]
\label{thm: ETP for general measure}
Let $N \in \bN$ and $0 < d \le 1/2$. For each integer $n \in [N,2N)$, we associate $A_n^*$ to $\bd = (d_1, \ldots, d_k)$ as in \S \ref{sec: rectangle}, where $\tau > 0$ is fixed and
\[
N^{- \tau} d \ll d_j \ll N^\tau d.
\]
Let $\mu$ be a Borel probability measure on $[0,1]^k$ with $\dim_{\ell^1}(\mu) = \kappa$. Then, for any fixed $\epsilon>0,$
\[
\tilde E_N(\mu)
\ll 
N^{k+(2k+\eps)\tau+\eps} d^{\kappa - \eps}.
\]
\end{thm}

\begin{proof}
Imitating the proof of Theorem \ref{thm: ETP} gives
\begin{align*}
&\frac{\tilde E_N(\mu)}{N^k d_1 \cdots d_k} \ll \sum_{\substack{n \mid \bxi \in \cR^\vee_N \setminus \{ \bzero \} \\ N \le n < 2N}} |\hat{\mu}(\bxi)|
+ \sum_{u=1}^\infty 2^{-uL} \sum_{\substack{n \mid \bxi \in 2^u \cR^\vee_N \setminus 2^{u-1} \cR^\vee_N 
 \\ N \le n < 2N}} 
|\hat{\mu}(\bxi)|,
\end{align*}
where
$
R_N^\vee = [-1/d_1,1/d_1] \times \cdots \times [-1/d_k, 1/d_k].
$
A similar argument to the remainder of that proof then delivers
\[
\frac{\tilde E_N(\mu)}{N^{(1+\tau)k} d^k} \ll N^\eps (N^\tau/d)^{k-\kap+\eps}.
\]
\end{proof}

\section{Variance transference principle}
\label{sec: VTP}
In this section, we establish the VTP.

\begin{thm}
\label{thm: VTP}
Let $k \in \bN$, let $\mu$ be a probability measure on $[0,1]^k$ satisfying \eqref{MainAssumption}, and let $(f_n)_{n=1}^\infty$ be an admissible function system. Then VTP holds for $(f_n)_n$ and $\mu$.
\end{thm}

\begin{proof}
We wish to estimate
\[
V_N(\mu) = \int_{[0,1]^k} \left(
\sum_{n \le N} (f_n(\bx) - \lam(f_n))
\right)^2 \d \mu(\bx).
\]
As $f_n$ is a Schwartz function, we may replace it by its Fourier series to obtain
\begin{align*}
\label{eq: VTP fourier identity}
V_N(\mu)
&= \int_{[0,1]^k} \left(
\sum_{n \le N} \sum_{\bxi \neq \bzero} \hat{f}_n(\bxi) e(\bxi \cdot \bx)
\right)^2 \d\mu(\bx) \\
&= \int_{[0,1]^k} \sum_{n,n' \le N} \sum_{\bxi,\bxi' \neq \mathbf{0}} \hat{f}_{n}(\bxi) \hat{f}_{n'}(\bxi') e((\bxi+\bxi') \cdot \bx) \d\mu(\bx)\\
&= \sum_{n,n'}
\sum_{\bxi,\bxi' \neq \mathbf{0}, \bxi+\bxi' = \mathbf{0}} \hat{f}_{n}(\bxi) \hat{f}_{n'}(\bxi')\\ & \qquad + \sum_{n,n'} \sum_{\bxi,\bxi',\bxi+\bxi'\neq\mathbf{0}}\hat{f}_{n}(\bxi)\hat{f}_{n'}(\xi') \hat{\mu}(\bxi+\bxi')
\\
&= V_N(\lam) + \sum_{n,n'} \sum_{\bxi,\bxi',\bxi+\bxi' \neq \mathbf{0}}\hat{f}_{n}(\bxi) \hat{f}_{n'}(\bxi') \hat{\mu}(\bxi+\bxi').
\end{align*}
The final equality is gotten by running the same calculation for $V_N(\lam)$. It is crucial that all of the Fourier series converge absolutely, which is the case as our functions are Schwartz. We are left with the contribution of the non-zero coefficients:
\[
V_N(\mu) = V_N(\lam) + E,
\]
where
\[
E = \sum_{n,n' \le N} \sum_{\bxi,\bxi',\bxi+\bxi' \neq \mathbf{0}} \hat{f}_{n}(\bxi) \hat{f}_{n'}(\bxi') \hat{\mu}(\bxi+\bxi').
\]

We sum over $n,n'$ dyadically, noting that
\[
|E| \le
\sum_{m,m' \le \frac{\log N}{\log 2}}
\sum_{\substack{n\in D_m \\ n'\in D_{m'}}} \sum_{\bxi, \bxi', \bxi+\bxi'\neq
\mathbf{0}} |\hat{f}_{n}(\bxi)\hat{f}_{n'}(\bxi')\hat{\mu}(\bxi+\bxi')|,
\]
where $D_m=[2^m,2^{m+1}).$ We normalise the Fourier coefficients by dividing by
\[
\lam(f_n) 
= \sum_{i\in I_m} n^k \prod_{j \le k} d_{i,j}
\asymp 2^{km} \sum_{i\in I_m} \prod_{j \le k} d_{i,j} =: w_m
\]
and 
$
\lam(f_{n'}) \asymp w_{m'}.
$
Writing 
\[
a_n(\bxi) = \frac{\hat{f}_n(\bxi)} {\lam(f_n)}, \qquad
a_{n'}(\bxi) = \frac{\hat{f}_{n'}(\bxi)} {\lam(f_{n'})},
\]
we now have
\begin{align*}
E &\ll
\sum_{m,m' \le \frac{\log N}{\log 2}} w_m w_{m'} \sum_{\substack{n\in D_m \\ n'\in D_{m'}}}
\sum_{\bxi, \bxi', \bxi + \bxi'\neq \bzero} |a_n(\bxi) a_{n'}(\bxi') \hat{\mu}(\bxi+\bxi')| \\
&= \sum_{m,m' \le \frac{\log N}{\log 2}} w_m w_{m'} S(m,m'),
\end{align*}
where
\[
S(m,m') =
\sum_{\substack{n \in D_m \\ n'\in D_{m'}}}
\sum_{\bxi+\bxi'\neq \bzero} |a_n(\bxi) a_{n'}(\bxi') \hat{\mu}(\bxi+\bxi')|.
\]

To proceed, we examine the coefficients $a_n,a_{n'},$ which satisfy 
\[
|a_n|, |a_{n'}| \le 1.
\]
The support of $f_n$ is essentially a union of rectangles with different side lengths but the same volume, and its Fourier series is the sum of the Fourier series of those rectangles. 
In what follows, we may assume that $I_m$ is non-empty, since otherwise it will not contribute anything to the sums. 

We write
\[
a_n = \sum_{i\in I_m} a_{n,i},
\]
where
$
\lam(f_n) a_{n,i} $
is the Fourier transform of $\bx \mapsto A_n^*(\bd^{(i)};\bx)$. Note that 
\[
\lam(f_n) \asymp w_m = (\# I_m) 2^{km} \prod_{j \le k} d_{i,j}
\]
for all $i \in I_m$, whence
\[
a_{n,i} \ll (\#I_m)^{-1}.
\]
We now see that
\[
S(m,m') \ll
\sum_{\substack{n\in D_m \\ n'\in D_{m'}}} \sum_{\substack{i \in I_m\\ i' \in I_{m'}}} \sum_{\bxi, \bxi', \bxi+\bxi'\neq \bzero} |a_{n,i}(\bxi)a_{n',i'}(\bxi') \hat{\mu}(\bxi+\bxi')|.
\]

The function $a_{n,i}$ is essentially supported on the multiples of $n$ in
\[
\cR_{m,i}^\vee := [-1/d_{i,1},1/d_{i,1}] \times \cdots \times [-1/d_{i,k},1/d_{i,k}].
\]
Indeed, we have the following estimate \cite[Lemma 5.3]{CY}.

\begin{lem} Let
\[
n \in D_m, \quad i \in I_m, \quad L,t \ge 1, \quad \bxi \in \bZ^k \setminus t \cR_{m,i}^\vee.
\]
Then
\[
a_{n,i}(\bx) \ll_L t^{-L} (\# I_m)^{-1}.
\]
\end{lem}

\noindent Similarly, the function $a_{n',i'}$ is essentially supported on the multiples of $n'$ in $\cR^\vee_{m',i'}$. 
The upshot is that
\[
S(m,m') \ll_L \sum_{u,u' = 0}^\infty S_{u,u'}(m,m'),
\]
where
\[
S_{u,u'}(m,m') =  \frac{2^{-L(u+u')}}{\#I_m \# I_{m'}} \sum_{\substack{i \in I_m \\ i' \in I_{m'}}} \sum_{\substack{\bxi, \bxi', \bxi + \bxi' \ne \bzero \\ n \mid \bxi \in 2^u \cR_{m,i}^\vee \\ n' \mid \bxi' \in 2^{u'} \cR_{m',i'}^\vee}} |\hat \mu(\bxi + \bxi')|.
\]
Here $L \ge 1$ is a constant that we will describe later.

Write
\[
T = \begin{pmatrix}
t_1 & t_1' \\
t_2 & t_2' \\
\vdots & \vdots \\
t_k & t_k'
\end{pmatrix}, \qquad
\bn = \begin{pmatrix} n \\ n' \end{pmatrix}, \qquad
\bx = \begin{pmatrix} x_1 \\ x_2 \\ \vdots \\ x_k \end{pmatrix}
\]
For $\bzero \ne \bx \in \bZ^k$, let $N(\bx) = N(\bx; m, m', u, u', i, i')$ count solutions $(\bt, \bt', \bn) \in \bZ^{2k+2}$ to
$T \bn = \bx$ such that
\begin{align*}
&n \in D_m, \qquad n' \in D_{m'}, \qquad \bt, \bt' \ne \bzero, \\
&|t_j n| \le 2^u/d_{i,j,m}, \quad
|t'_j n'| \le 2^{u'}/d_{i',j,m'}
\qquad (1 \le j \le k).
\end{align*}
Then
\begin{align*}
S_{u,u'}(m,m') = \frac{2^{-L(u+u')}}{\# I_m \# I_{m'}} \sum_{\substack{i\in I_m
\\ i'\in I_{m'}}} \sum_{\bzero \ne \bx \in \mathbb{Z}^k} |\hat{\mu}(\bx)| N(\bx).
\end{align*}
This is a finite sum, since $N(\bx) = 0$ for large $\bx$. Indeed, if $N(\bx) \ge 1$ then
\begin{equation}
\label{xbound}
|x_j| \le \frac{2^u}{d_{i,j,m}} + \frac{2^{u'}}{d_{i',j,m'}} \ll \frac{2^{u + \tau m}}{\fd_m} + \frac{2^{u' + \tau m'}}{\fd_{m'}}.
\end{equation}

For $j = 1,2,\ldots, k$, we can bound the number of possibilities for $t_j, t_j'$ by
\[
r_j := \frac{2^{u+2-m}}{d_{i,j,m}} + 1, \qquad r_j' := \frac{2^{u'+2-m'}}{d_{i',j,m'}} + 1, 
\]
respectively. Note that
\[
r_j \ll \frac{2^{u + (\tau - 1) m}}{\fd_m},
\qquad
r_j' \ll \frac{2^{u' + (\tau-1)m'}}{\fd_{m'}}.
\]
To estimate $N(\bx)$, we decompose $N(\bx) = N_1(\bx) + N_2(\bx)$, where $N_1(\bx)$ counts solutions with $\bt, \bt'$ parallel. Without loss of generality, we assume that $m \le m'$.

\subsection{The partial-rank case}

Since $\bx \ne \bzero$ we may, by symmetry, assume that $x_1 \ne 0$. As $n \bt + n' \bt' = \bx$ for any solution counted by $N_1(\bx)$, the vectors $\bt$ and $\bt'$ must be multiples of $\bx$, with $t_1 t_1' \ne 0$. The count $N_1(\bx)$ is therefore bounded by the number of solutions to
\[
t_1 n + t_1' n' = x_1.
\]
Choosing the value of $t_1 n \notin \{ 0,x_1 \}$, then applying the divisor bound to determine $t_1, n, t_1', n'$, gives
\[
N_1(\bx) \ll 2^{\tau m} \left( \frac{2^{u + \tau m}}{\fd_m} + \frac{2^{u' + \tau m'}}{\fd_{m'}} \right).
\]

\subsection{The full-rank case}

Since $\bt,\bt'$ are not parallel we may, be symmetry, assume that the first two rows of $T$ are linearly independent. As $(t_1, t_2) \ne (0,0)$, we may also assume that $t_1 \ne 0$. We begin by choosing $t_1 \ne 0$, as well as $t'_1$ and $t_2$. Using that
\[
t_1 n + t'_1 n' = x_1, \qquad t_2 n + t'_2 n' = x_2,
\]
we see that
\[
(t_2 t'_1 - t_1 t'_2) n' = t_2 x_1 - t_1 x_2.
\]
The left hand side is non-zero, so
\[
n' \mid t_2 x_1 - t_1 x_2 \neq 0.
\]
Thus, by the divisor bound, the number of possibilities for $n'$ is at most $O(2^{(m+m')\tau}).$ After choosing $n'$, we can determine at most one possibility for $n$ via
\[
t_1 n + t'_1 n' = x_1,
\]
and at most one possibility for $t'_2$ via
\[
t_2 n + t'_2 n' = x_2.
\]
Choosing $t_3,\dots,t_k$ will then determine $t'_3,\dots,t'_k$ in at most one way. The upshot is that
\[
N_2(\bx) \ll r_1 r_1' r_2 r_3 \cdots r_k 2^{(m+m')\tau}
\ll  \left(\frac{2^{u + (\tau - 1) m}}{\fd_m} \right)^k \frac{2^{u' + (\tau-1)m'}}{\fd_{m'}}
2^{(m+m')\tau}.
\]

\subsection{VTP conclusion}

Combining our bounds on $N_1(\bx)$ and $N_2(\bx)$ with \eqref{xbound} furnishes
\begin{align*}
&\sum_{\bx \ne \bzero} |\hat \mu(\bx)| N(\bx) \ll 2^{O(\tau)m'}
\left( \frac{2^{u}}{\fd_m} + \frac{2^{u'}}{\fd_{m'}} \right)^{k - \dim_{\ell^1}(\mu)} \left(\frac{2^{u - m}}{\fd_m} \right)^k \frac{2^{u'-m'}}{\fd_{m'}}.
\end{align*}
Now, taking $L \ge 1$ sufficiently large,
\[
S(m,m') \ll 2^{O(\tau)m'}
\left( \frac{1}{\fd_m} + \frac{1}{\fd_{m'}} \right)^{k - \dim_{\ell^1}(\mu)} \left(\frac{2^{- m}}{\fd_m} \right)^k \frac{2^{-m'}}{\fd_{m'}}.
\]
Therefore
\begin{align*}
E &\ll \sum_{m, m' \le \frac{\log N}{\log 2}} w_m w_{m'} \\
&\qquad \cdot \left(2^{O(\tau)m'}
\left( \frac{1}{\fd_m} + \frac{1}{\fd_{m'}} \right)^{k - \dim_{\ell^1}(\mu)} \left(\frac{2^{- m}}{\fd_m} \right)^k \frac{2^{-m'}}{\fd_{m'}} \right) \\
&\ll \sum_{m, m' \le \frac{\log N}{\log 2}} w_m w_{m'}  2^{O(\tau)m'}
(2^{m'(k+1)/k})^{k - \dim_{\ell^1}(\mu)} 2^m 2^{m'/k}.
\end{align*}

Since
\[
E_N(\lam) \asymp \sum_{m \le \frac{\log N}{\log 2}} 2^m w_m
\]
whenever $N$ is a power of two, and since $\tau$ is small, it remains to show that
\[
(1 + 1/k) (k - \dim_{\ell^1}(\mu)) + 1/k < 1.
\]
This is nothing more than a rearrangement of \eqref{MainAssumption}.
\end{proof}

\section{Completing the proofs}\label{sec: final proofs}

In this section, we establish Theorems \ref{thm: convergence}, \ref{thm: divergence},  \ref{HausdorffConvThm}, \ref{thm: Rational Counting}, \ref{OtherBound}, and \ref{thm: Intrinsic App}.

\subsection{Proof of Theorem \ref{thm: convergence}}
\label{ConvergentProofs}

For the convergent Khinchin statement, we may assume that
\[
\psi(n) \ge (n \log^2 n)^{-1/k}
\qquad (n \ge 2).
\]
Indeed, we may replace $\psi$ by $\tilde \psi$, where $\tilde \psi(n) = \psi(n) + (n \log^2 n)^{-1/k}$ for $n \ge 2$, since $\psi \le \tilde \psi$ and $\sum_n \tilde \psi(n)^k < \infty$.

We use $f_n(\bx) = A_n^*(\bx; \bd)$ from \eqref{AnStar} where, for $n \in D_m$,
\[
\bd = (d, \ldots, d), \qquad
d = \psi(2^{m-1})/2^{m-1}.
\]
Then
$
\lam(f_n) \ll \psi(2^{m-1})^k
$
so, by the Cauchy condensation test and the convergence of $\sum_{n=1}^\infty \psi(n)^k$, we have \eqref{LebesgueConvergence}. Moreover, we have $f_n \gg 1$ on
\[
\{ \bx \in [0,1)^k:
\max \{
\| n x_1 - y_1 \|, \ldots,
\| n x_k - y_k \|
\} < \psi(n) \}.
\]
By Lemma \ref{GenConv}, it remains to confirm the ETP. 

For any $\eps > 0$, Theorem \ref{thm: ETP for general measure} gives
\[
\sum_{n \in D_m} (\mu(f_n) -\lam(f_n)) \ll_\eps (2^{m-1})^{k+\eps} (\psi(2^{m-1})/2^{m-1})^{\kap - \eps}.
\]
For $N = 2^M - 1$, we thus have
\begin{align*}
E_N(\mu) - E_N(\lam) &\ll 
\sum_{m \le M-1} (2^{m-1})^{k - \kap + 2 \eps} \psi(2^{m-1})^{\kap - \eps},
\end{align*}
where $\kap = \dim_{\ell^1}(\mu)$.
On the other hand,
\[
E_N(\lam) \gg \sum_{n \le N} \psi(n)^k \ge \sum_{m \le M-1} 
2^{m-1} \psi(2^{m-1})^k.
\]

We see from \eqref{WeakAssumption} that if $\eps$ is sufficiently small then
\[
n^{k - \kap  - 1 + 2 \eps}
= o((n \log^2 n)^{(\kap - \eps -k)/k})
\]
and hence 
\[
n^{k - \kap  - 1 + 2 \eps} = o(\psi(n)^{\kap - \eps - k}).
\]
Applying this to $n = 2^{m-1}$, we now have
\begin{align*}
E_N(\mu) - E_N(\lam) &\ll \sum_{m \le M - 1} o_{m \to \infty}(2^{m-1} \psi(2^{m-1})^k) \\
&= O(1) + o (E_N(\lam)),
\end{align*}
as $N = 2^M - 1 \to \infty$. This verifies the ETP and completes the proof of the Khinchin assertion.

\bigskip

For the convergent Gallagher statement, we may assume that
\[
\psi(n) \ge (n \log^{k+1} n)^{-1} \qquad (n \ge 2).
\]
Indeed, we may replace $\psi$ by $\tilde \psi$, where $\tilde \psi(n) = \psi(n) + (n \log^{k+1} n)^{-1}$ for $n \ge 2$, since $\sum_n \tilde \psi(n) \log^{k-1} n < \infty$.

We deploy the rectangular decomposition $\fC_n$ from \S \ref{sec: rectangle}. Our smooth approximation then satisfies $f_n \gg 1$ on $A_n^\times$ and, since
\[
\lam(f_n) \ll \psi(2^{m-1}) \log^{k-1}(1/\psi(2^{m-1})) 
\ll \psi(2^{m-1}) m^{k-1} \qquad (n \in D_m),
\]
we have \eqref{LebesgueConvergence} by the Cauchy condensation test and the convergence of $\sum_{n=1}^\infty \psi(n) (\log n)^{k-1}$. By Lemma \ref{GenConv}, it remains to confirm the ETP. Note that $f_n$ is a linear combination of $\asymp \log^{k-1}(1/\psi(n))$ many functions $A_n^* = A_n^*(\bd; \cdot)$, where
\[
\frac{\psi(2^{m-1})}{2^{m-1}} \le d_j \le \frac1{2^{m-1}} \quad(1 \le j \le k), \qquad
d_1 \cdots d_k \asymp \frac{\psi(2^{m-1})}{2^{k(m-1)}}.
\]
It therefore suffices to prove that
\[
\sum_{n \in D_m} (\mu(A_n^*) - \lam(A_n^*)) = 
o(2^{m-1} m^{k-1} \psi(2^{m-1})).
\]

Recall that $\mu = \mu_1 \times \cdots \times \mu_k$, where each $\mu_j$ is a missing-digit measure. By \eqref{SplitAssumption}, there exists $\kap$ in the range
\[
\min \{ \dim_{\ell^1}(\mu_j): 1 \le j \le k \} > \kap > 1 - \frac{1}{k+1}.
\]
For any fixed $\eps > 0$, Theorem \ref{thm: ETP} gives
\begin{align*}
\sum_{n \in D_m} (\mu(A_n^*) -\lam(A_n^*)) &\ll (2^{m-1})^{k+\eps} \sum_{i \in I_m} (d_1 \cdots d_k)^{\kap} \\
&\ll (2^{m-1})^{k+\eps} \left(
\frac{\psi(2^{m-1})}{2^{k(m-1)}} \right)^{\kap}.
\end{align*}

Writing $n=2^{m-1}$, it remains to prove that if $\eps$ is sufficiently small then
\[
n^{k + \eps} (\psi(n)/n^k)^{\kap} = o(n \psi(n)).
\]
This is equivalent to
\[
n^{k(1-\kap) - 1 + \eps} = o(\psi(n)^{1-\kap}),
\]
which follows from the inequalities
\[
\psi(n) \ge \frac1{n^{1+\eps}},
\qquad
1 - \kap < \frac1{k+1}
\]
and the fact that $\eps$ is small. This completes the proof of Theorem \ref{thm: convergence}.

\subsection{Proof of Theorem \ref{thm: divergence}}

For the divergent Khinchin statement, we use an admissible function system $(f_n)_{n=1}^\infty$ built from the admissible set system in Example \ref{SimOK}. Then $\lam(f_n) \gg \psi(2^m)^k$ so, by the Cauchy condensation test and the divergence of $\sum_{k=1}^\infty \psi(n)^k$, we have \eqref{LebesgueDivergence}. For any $C > 1$, we may apply Theorem \ref{LebesgueEstimate} to obtain \eqref{LebesgueQIA}. We have the ETP, in the same way as for the convergent Khinchin statement (see \S \ref{ConvergentProofs}). We also have VTP, by Theorem \ref{thm: VTP}. Lemma \ref{GenDiv} now completes the proof of the divergent Khinchin assertion.

\bigskip

For the divergent Gallagher statement, we use an admissible function system $(f_n)_{n=1}^\infty$ built from the admissible set system in Example \ref{MultOK}. Then $\lam(f_n) \gg \psi(2^m) m^{k-1}$ so, by the Cauchy condensation test and the divergence of the series $\sum_{k=1}^\infty \psi(n) (\log n)^{k-1}$, we have \eqref{LebesgueDivergence}. For any $C > 1$, we may apply Theorem \ref{LebesgueEstimate} to obtain \eqref{LebesgueQIA}. We also have the VTP, by Theorem \ref{thm: VTP}. By Lemma \ref{GenDiv}, it remains to confirm the ETP.

By \eqref{MainAssumption}, there exists $\kap$ in the range
\[
\dim_{\ell^1}(\mu) > \kap > k - \frac{k-1}{k+1}.
\]
For any fixed $\eps > 0$, Theorem \ref{thm: ETP for general measure} gives
\begin{align*}
\sum_{n \in D_m} (\mu(A_n^*) -\lam(A_n^*))
&\ll (2^{m-1})^{k+O(\tau+\eps)} 
\left(
\frac{\psi(2^{m-1})^{1/k}}{2^{(m-1)}} \right)^{\kap}.
\end{align*}
Writing $n=2^{m-1}$, it suffices to prove that if $\eps$ is sufficiently small then
\[
n^{k + \eps} (\psi(n)^{1/k}/n)^{\kap} = o(n \psi(n)).
\]
This is equivalent to
\[
n^{k - 1 - \kap + \eps} = o(\psi(n)^{(k-\kap)/k}),
\]
which follows from the inequalities
\[
\psi(n) \ge \frac1{n^{1+\eps}},
\qquad
\kap > \frac{k^2 + 1}{k+1} > \frac{k^2}{k+1}
\]
and the fact that $\eps$ is small. This completes the proof of Theorem \ref{thm: divergence}.

\subsection{Proof of Theorem \ref{HausdorffConvThm}}

The proof is similar to that of the Khinchin part of Theorem \ref{thm: convergence}, but we now exploit the quantitative error term in Theorem \ref{thm: ETP for general measure}. Let $\mu$ be the Cantor--Lebesgue measure of $\cK$. Note that $\mu$ is $\kap_2$-regular, where $\kap_2 = \Haus(\cK)$. 

Our setup is similar to that of the Khinchin part of Theorem \ref{thm: convergence}, with $\psi = \psi_t$. For $n \in D_m$, the smooth function $f_n$ is supported on a union of balls of radius $2 \psi(2^{m-1})/2^{m-1}$, and
$f_n \gg 1$ on at set $\cS_n$ containing
\[
A_n := 
\{ \bx \in [0,1)^k: \max \{
\| n x_1 - y_1 \|, \ldots, \| n x_k - y_k \| \} < \psi(n) \}.
\]
Moreover, if $\bz \in A_n$ then $B_r(\bz) \subseteq \cS_n$, where 
\[
r = r_m = \psi(2^{m-1})/2^m.
\]
As $\mu$ is $\kap_2$-regular, we now see that $\bigcup_{n \in D_m} A_n \cap \cK$ is covered by a collection $\cB_m$ of hypercubes $B$ of diameter $|B| \asymp r$
centred in $\cK$ such that
\[
\sum_{B \in \cB_m} \mu(B) \ll \sum_{n \in D_m} \mu(\cS_n) \ll \sum_{n \in D_m} \mu(f_n).
\]
Thus, by Theorem \ref{thm: ETP for general measure} and the $\kap_2$-regularity of $\mu$, we have
\begin{align*}
\label{eqn: rational counting}
r_m^{\kap_2} 
\# \cB_m \ll_\eps 2^{m-1} \psi(2^{m-1})^k + (2^{m-1})^{k + \eps} r_m^{\kap - \eps}
\end{align*}
for any constant $\eps > 0$,
where $\kap = \dim_{\ell^1}(\mu)$.

Now, for any fixed $s > 0$,
\[
\sum_{B \in \cB_m} |B|^s \ll r_m^s \# \cB_m \ll r_m^{s-\kap_2} (2^{m-1} \psi(2^{m-1})^k + (2^{m-1})^{k + \eps} r_m^{\kap - \eps}).
\]
By Theorem \ref{LevesleyThm} and Lemma \ref{HausdorffCantelli}, it remains to show that
\begin{equation}
\label{FirstTermConverges}
\sum_{m=1}^\infty 
2^{m-1} r_m^{s-\kap_2} \psi(2^{m-1})^k < \infty
\end{equation}
and
\begin{equation}
\label{FirstTermDominates}
(2^{m-1})^{k + \eps} r_m^{\kap - \eps} \ll 2^{m-1} \psi(2^{m-1})^k,
\end{equation}
whenever $\eps$ is sufficiently small and 
\[
\frac1k  < t < \frac1k + \eps,
\qquad
s > \frac{k+1}{t+1} + \kap_2 - k.
\]

By the Cauchy condensation test, we have \eqref{FirstTermConverges} if and only if
\[
\sum_{n=1}^\infty (\psi(n)/n)^{s-\kap_2} \psi(n)^k < \infty.
\]
The latter indeed holds, since
\[
(\psi(n)/n)^{s-\kap_2} \psi(n)^k =
n^{(t+1)(\kap_2 - s)-tk}
\]
and
\begin{align*}
(t+1)(\kap_2 - s)-tk < 
(t+1) \left( k - \frac{k+1}{t+1} \right) - tk = -1.
\end{align*}

With $n = 2^{m-1}$, we can rewrite \eqref{FirstTermDominates} as
\[
n^{k+\eps} (\psi(n)/n)^{\kap - \eps} \ll n \psi(n)^k,
\]
which is equivalent to
\[
\psi(n)^{k - \kap + \eps} \gg n^{k-1+2\eps - \kap}.
\]
By \eqref{WeakAssumption}, the latter holds for any sufficiently small $\eps$, since $\psi(n) = n^{-t}$ and
\[
\frac{\kap - k}{k} > k - 1 - \kap.
\]
This confirms \eqref{FirstTermDominates} and completes the proof of Theorem \ref{HausdorffConvThm}.

\subsection{Proof of Theorem \ref{thm: Rational Counting}}
\label{RationalCountingProof}

Let us write
\[
\kap = \dim_{\ell^1}(\mu), \qquad \kap_2 = \Haus(\cK).
\]
For the upper bound we use the smooth function $f_n(\bx) = A_n^*(\bx; \bd)$ from \eqref{AnStar}, where
\[
\bd = (\del/N, \ldots, \del/N), \qquad
 \by = \bzero,
\]
for each $n \in [N,2N)$ and each $N \in \bN$ such that $N \le Q$ and $Q/N$ is a power of two. By Theorem \ref{thm: ETP for general measure}, if $\eps > 0$ then
\[
\sum_{N \le n < 2N} \mu(f_n) \ll_\eps \del^k N + N^{k+\eps} (\del/N)^{\kap - \eps}.
\]
To each $(\ba, n) \in \cN_\cK(2N-1, \del) \setminus \cN_\cK(N-1, \del)$ we may associate a hypercube $B$ of diameter $|B| \asymp \del/N$ centred in $\cK$, such that
\[
\sum_B \mu(B) \ll \sum_{N \le n < 2N} \mu(f_n).
\]
Using the $\kap_2$-regularity of $\mu$, we thus obtain
\begin{align*}
&\# \cN_\cK(2N-1, \del) - \# \cN_\cK(N-1, \del) \\ & \qquad \ll (\del/N)^{-\kap_2} (\del^k N + N^{k+\eps} (\del/N)^{\kap - \eps}) \\
& \qquad = \del^{k-\kap_2} N^{\kap_2 + 1} + \del^{\kap-\kap_2-\eps} N^{\kap_2 + k - \kap + 2\eps},
\end{align*}
and summing over $N$ yields
\[
\# \cN_\cK(Q, \del) \ll \del^{k-\kap_2} Q^{\kap_2 + 1} + \del^{\kap-\kap_2-\eps} Q^{\kap_2 + k - \kap + 2\eps}.
\]
Now taking $\eps, \eta$ small, we see from \eqref{WeakAssumption} that if $\del \gg Q^{-\eta-1/k}$ 
then $\# \cN_\cK(Q, \del) \ll \del^{k-\kap_2} Q^{\kap_2 + 1}$.

For the lower bound we use $f_n(\bx) = A_n^*(\bx; \bd)$ from \eqref{AnStar}, with bump function $w$ supported on $\{ 1/2 \le |x| \le 1 \}$, where
\[
\bd = (\del/Q, \ldots, \del/Q), \qquad
 \by = \bzero,
\]
for each $n \in [N,2N)$. Here $N = \lfloor (Q-1)/2 \rfloor$. By Theorem \ref{thm: ETP for general measure} and \eqref{WeakAssumption}, if $\eps > 0$ is a sufficiently small constant then
\[
\sum_{N \le n < 2N} \mu(f_n) \ge \eps \del^k N - \eps^{-1} N^{k+\eps} (\del/N)^{\kap - \eps} \gg \del^k Q.
\]
Finally, by \eqref{MeasureUpper},
\[
\# \cN_\cK(Q, \del) \gg (\del/Q)^{-\kap_2}  \del^k Q = \del^{k - \kap_2} Q^{\kap_2 + 1}.
\]
This completes the proof of Theorem \ref{thm: Rational Counting}.

\subsection{Proof of Theorem \ref{OtherBound}}

We take $\cK = \cK_{b,\cD}$ with
\[
\cD = \{ 0, 1, \ldots, b-1 \}^{k-1} \times \{ 0, 1, \ldots, a-1 \},
\]
for some positive integers $a < b$ to be chosen. We write
\[
\kap = \dim_{\ell^1}(\mu), \qquad \kap_2 = \Haus(\cK).
\] 
We may assume that $\eta < 1/100$. By Theorem \ref{thm: l1 bound}(b), if $b$ is sufficiently large then
\[
\kap_2 - \kap \le \frac{k \log (2 \log b)}{\log b} < \eta^3 =: \eps.
\]
By \eqref{HausdorffMissing}, if $b$ is sufficiently large then there exists an integer $a \in [1,b-1]$ such that
\[
k - \kap_2 = 1 - \frac{\log a}{\log b} \in (\eta/2,\eta).
\]
Now that we have provided the construction, we proceed to estimate $\# \cN_{\cK}(Q,0)$.

Let $\mu$ be the Cantor--Lebesgue measure of $\cK$.
With $\del \in (0,1/2]$ chosen in due course, we use the smooth function $f_n(\bx) = A_n^*(\bx; \bd)$ from \eqref{AnStar}, where
\[
\bd = (\del/N, \ldots, \del/N), \qquad
 \by = \bzero,
\]
for each $n \in [N,2N)$ and each $N \in \bN$ such that $N \le Q$ and $Q/N$ is a power of two. By Theorem \ref{thm: ETP for general measure}, 
\[
\sum_{N \le n < 2N} \mu(f_n) \ll \del^k N + N^{k+\eps} (\del/N)^{\kap - \eps}.
\]
Reasoning as in \S \ref{RationalCountingProof}, we obtain
\begin{align*}
\# \cN_{\cK}(Q, \del) &\ll
\del^{k-\kap_2}Q^{\kap_2+1} + \del^{\kap-\kap_2-\eps} Q^{\kap_2 + k - \kap + 2 \eps}.
\end{align*}
The choice $\delta = Q^{-2/\eta}$ delivers
$
\# \cN_{\cK}(Q, \del) \ll Q^{k+\eta},
$
which finishes the proof because
$
\cN_\cK(Q,0)\leq \cN_\cK(Q,\delta). 
$

\subsection{Proof of Theorem \ref{thm: Intrinsic App}}

Observe that
\[
W_k^\cK(\tau) = \limsup_{m \to \infty} A_m,
\]
where
\[
A_m = \bigcup_{\substack{\ba/n \in \cK \\ n \in D_m}} \prod_{j \le k}
\left(
\frac{a_j - n^{-\tau}}{n},
\frac{a_j + n^{-\tau}}{n}
\right) \qquad (m \in \bN).
\]
Put $\kap_2 = \Haus(\cK)$. By Corollary \ref{CountingRationalPoints} and the $\kap_2$-regularity of $\mu$, there exists a constant $\eta > 0$ such that
\[
\mu(A_m) \ll (2^m)^{(1+1/k) \kap_2 -\eta - (1 + \tau) \kap_2}.
\]
If $\tau$ is sufficiently close to $1/k$, then
$
(1+1/k)\kap_2 -\eta - (1 + \tau)\kap_2
< 0,
$
so
\[
\sum_{m=1}^\infty \mu(A_m) < \infty.
\]
By the first Borel--Cantelli lemma, we now have $\mu(W_k^\cK(\tau)) = 0$. This completes the proof of Theorem \ref{thm: Intrinsic App}.

\appendix

\section{More about the Fourier \texorpdfstring{$\ell^1$}{l1} dimension}
\label{sec: morel1}

In this appendix, we further discuss Fourier $\ell^1$ dimension for missing-digit measures.

\subsection{The Fourier \texorpdfstring{$\ell^1$}{l1} dimension in special cases}
\label{sec: special}

In a natural, wide class of special cases, we can rigorously estimate the Fourier $\ell^1$ dimension. Building on \cite[Theorem 2.6]{CVY}, the following estimates were demonstrated in \cite[Theorem 2.15]{YuManifold}. 

\begin{thm}
[Yu \cite{YuManifold}]
\label{thm: l1 bound}
Let $k \in \bN$, and let $\mu_{b,\cD}$ be the Cantor--Lebesgue measure of a missing-digit fractal $\cK_{b,\cD}$ on $[0,1]^k$.

\begin{enumerate}[(a)]
\item Let $t \in \bN$. Then
\[
\liminf_{\substack{b\to\infty \\ \#\cD \geq b^k-t}} \dim_{\ell^1}(\mu_{b,\cD}) = k.
\]
In particular, if $\eps > 0$ and $b$ is sufficiently large, then 
\[
\dim_{\ell^1}(\mu_{b,D}) > k-\epsilon
\]
holds whenever 
$\#\cD = b^k -1.$ 
\item Let $b\geq 4$ be an integer, and suppose $\cD$ has the form
\[
\cD = [a_1,b_1] \times [a_2,b_2] \times \cdots \times [a_k,b_k] \cap 
\{0,\dots,b-1\}^k.
\]
Then 
\[
\dim_{\ell^1} (\mu_{b,\mathcal{D}}) \geq \Haus(\cK_{b,\cD}) -\frac{k\log (2\log b)}{\log b}.
\]
\end{enumerate}
\end{thm}

\subsection{Estimating the Fourier \texorpdfstring{$\ell^1$}{l1} dimension}
\label{sec: compute}

Here we provide a general algorithm to estimate the Fourier $\ell^1$ dimension of a missing-digit measure $\mu = \mu_{b, \bP}$ on $[0,1]^k.$ The strategy comes from \cite{CVY}. For $\bx \in \bR^k$, define
\[
g(\bx) =  \left| \sum_{\bd \in \{0,1,\ldots,b-1\}^k} \bP(\bd) e(\bd \cdot \bx) \right|.
\]
Since the Fourier transform of the distribution of a sum of independent random variables is the product of the Fourier transforms of the individual distributions, we have
\begin{equation}
\label{ProductFormula}
|\hat \mu(\bx)| = \prod_{j=1}^\infty g(b^{-j} \bx).
\end{equation}
For $L \in \bN$ and $\bx \in \bR^k$, define
\[
S_L(\bx) = \prod_{j=0}^{L-1} g(b^j \bx).
\]

\begin{thm} 
\label{l1lower}
If $L \in \bN$ then
\[
\dim_{\ell^1}(\mu) \ge \frac{-\log \left(
\displaystyle \sup_\bx b^{-kL} \sum_{i_1, \ldots, i_k = 0}^{b^L-1} S_L\left(\bx + \frac{\bi}{b^L} \right)
\right)}{\log (b^L)}.
\]
\end{thm}

\begin{proof} We may assume without loss of generality that $L = 1$. Indeed, if $L \ge 2$, then we can replace $b$ by $b^L$ and $\bP$ by the distribution of
\[
\sum_{j=0}^{L-1} \bd^{(j)} b^j,
\]
where the $\bd^{(j)}$ are independent random variables distributed according to $\bP$. This leaves $\mu$ unchanged, and replaces $S_L$ by $S_1 = g$.

It now suffices to prove that
\[
\sum_{\| \bxi \|_\infty \le Q} |\hat \mu(\bxi)| \ll Q^{k-s} \qquad (Q \in \bN),
\]
where 
\[
s = \frac{-\log \left(
\displaystyle \sup_\bx b^{-k} \sum_{i_1, \ldots, i_k = 0}^{b-1} g\left(\bx + \frac{\bi}{b} \right)
\right)}{\log b}.
\]
Furthermore, we may assume that $Q = b^N - 1$ for some $N \in \bN$. Thus, it remains to show that
\[
\sum_{\| \bxi \|_\infty < b^N} |\hat \mu(\bxi)| \ll \left(
\displaystyle \sup_\bx \sum_{i_1, \ldots, i_k = 0}^{b-1} g\left(\bx + \frac{\bi}{b} \right) \right)^N.
\]

Arguing by induction, we establish the sharper and more general inequality
\[
\displaystyle \sup_\by \sum_{\| \bxi \|_\infty < b^N} |\hat \mu(\by + \bxi)| \le \left( \sup_\bx \sum_{i_1, \ldots, i_k = 0}^{b-1} g\left(\bx + \frac{\bi}{b} \right) \right)^N,
\]
for every integer $N \ge 0$.
The case $N = 0$ is simply the fact that $|\hat \mu(\by)| \le 1$ for all $\by$. 

Let us now take $N \ge 1$, and assume the inequality with $N-1$ in place of $N$. By \eqref{ProductFormula},
\[
|\hat \mu(\bx)| = g(b^{-1}\bx) |\hat \mu(b^{-1}\bx)| \qquad (\bx \in \bR^k).
\]
Writing $\bxi = \bxi^{(1)} + b \bxi^{(2)}$ now gives
\begin{align*}
&\sum_{\| \bxi \|_\infty < b^N} |\hat \mu(\by + \bxi)| \\ &= \sum_{\| \bxi^{(1)} \|_\infty < b} g(b^{-1} \by + b^{-1} \bxi^{(1)}) \sum_{\| \bxi^{(2)} \|_\infty < b^{N-1}} |\hat \mu(b^{-1} \by + b^{-1} \bxi^{(1)} + \bxi^{(2)})|.
\end{align*}
Thus, by the inductive hypothesis,
\begin{align*}
&\sum_{\| \bxi \|_\infty < b^N} |\hat \mu(\by + \bxi)| \\ &\le \sum_{\| \bxi^{(1)} \|_\infty < b} g(b^{-1} \by + b^{-1} \bxi^{(1)}) \left( \sup_\bx \sum_{i_1, \ldots, i_k = 0}^{b-1} g\left(\bx + \frac{\bi}{b} \right) \right)^{N-1} \\
&\le \left( \sup_\bx \sum_{i_1, \ldots, i_k = 0}^{b-1} g\left(\bx + \frac{\bi}{b} \right) \right)^N.
\end{align*}
This closes the induction and completes the proof.
\end{proof}

We expect that replacing the maximum with a minimum would give rise to an upper bound. However, we also expect that a proof of this would be technical and lengthy. Moreover, the lower bound is what is needed for applications. For these reasons, we do not establish the upper bound.

Note by periodicity that the lower bound in Theorem \ref{l1lower} is effectively computable.

\bibliographystyle{amsplain}

\end{document}